\documentclass{amsart}

\usepackage[all]{xy}
\usepackage{hyperref}
\hypersetup{
    colorlinks=true,       % false: boxed links; true: colored links
    linkcolor=blue,          % color of internal links
    citecolor=blue,        % color of links to bibliography
    filecolor=blue,      % color of file links
    urlcolor=blue           % color of external links
}
\usepackage[backrefs,lite]{amsrefs}

\newcommand{\Pic}{\operatorname{Pic}}

\newcommand{\vdim}{\operatorname{vdim}}

\newcommand{\edim}{\operatorname{edim}}
\newcommand{\efdim}{\operatorname{efdim}}

\newcommand{\fdim}{\operatorname{fdim}}

\newcommand{\std}{\operatorname{std}}

\newcommand{\seep}{\operatorname{sp}}

\newcommand{\Bs}{\operatorname{Bs}}

\newcommand{\pp}{\mathbb{P}}

\newcommand{\zz}{\mathbb{Z}}
\newcommand{\nn}{\mathbb{N}}

\newcommand{\cc}{\mathbb{C}}
\newcommand{\kk}{\mathbb{K}}
\newcommand{\ls}{\mathcal{L}}

\newtheorem{theorem}{Theorem}[section]
\newtheorem{lemma}[theorem]{Lemma}
\newtheorem{proposition}[theorem]{Proposition}
\newtheorem{corollary}[theorem]{Corollary}
\newtheorem{conjecture}[theorem]{Conjecture}

\theoremstyle{definition}

\newtheorem{definition}[theorem]{Definition}
\newtheorem{example}[theorem]{Example}

\newtheorem{remark}[theorem]{Remark}

\theoremstyle{remark}

\numberwithin{equation}{section}

\usepackage[norelsize]{algorithm2e}

\begin{document}

\title{Linear Systems on the Blow-Up of $(\pp^1)^n$}
\author[A.~Laface]{Antonio Laface}
\address{
Departamento de Matem\'atica,
Universidad de Concepci\'on,
Casilla 160-C,
Concepci\'on, Chile}
\email{alaface@udec.cl}

\author[J.~Moraga]{Joaqu\'in Moraga}
\address{
Departamento de Matem\'atica,
Universidad de Concepci\'on,
Casilla 160-C,
Concepci\'on, Chile}
\email{joamoraga@udec.cl}

\subjclass[2010]{Primary 14C20, 
Secondary 14M25. 
}

\thanks{
The authors were partially supported 
by Proyecto FONDECYT Regular N. 1150732.
}

\maketitle

\medskip 

\begin{abstract}
In this note we study linear systems
on the blow-up of $(\pp^1)^n$ 
at $r$ points in very general position. 
We prove that the fibers of the projections 
$(\pp^1)^n \rightarrow (\pp^1)^s, 1\leq s \leq n-1$
can give contribution to the speciality of the linear system. 
This allows us to give a new definition
of expected dimension
of a linear system in $(\pp^1)^n$
which we call fiber dimension.
Finally, we state a conjecture
about linear systems on $(\pp^1)^3$.
\end{abstract}

\section*{Introduction}
An open problem in algebraic geometry
is that of determining the dimension of
a linear system of hypersurfaces of $\pp^n$
of a given degree passing through finitely 
many points in very general position
with prescribed multiplicities.
This problem is related to polynomial
interpolation in several variables
to the Waring problem for polynomials
and the classification of defective
higher secant varieties of Veronese
embeddings of projective spaces
~\cite[Sections 6 and 7]{C}.
In case $n=2$ the Segre-Harbourne-Gimigliano-Hirschowitz
conjecture~\cite{Gi,Ha,Hi,Se} predicts
the dimension of such linear systems. 
Several cases of this conjecture 
have been proved, see e.g.~\cite{C,CHMR,CM1,CM2,Le,Mig}.
In ~\cite{LU2,LU1} an analogous conjecture
is stated for $n=3$ and proved when the
multiplicities of the points are $\leq 5$ in
~\cite{BB,BBCS}. 
There is no such a conjecture for higher
values of $n$, but all the same there are 
partial results about the dimension of such linear 
systems~\cite{Dum,St} and
in ~\cite{BDP} the authors determine the
contribution to the dimension of linear
systems given by linear subspaces.

Inspired by~\cite{BDP} in this note 
we study linear systems of 
$(\pp^1)^n$ through multiple points
for $n\geq 2$.
Let $\ls:=\ls_{(d_1, \dots , d_n)}(m_1, \dots , m_r)$
be the linear system of hypersurfaces of degree $(d_1, \dots , d_n)$
in $(\pp^1)^n$ passing through a general 
union of $r$ points with multiplicities respectively $m_1, \dots m_r$.
The {\em virtual dimension} of $\ls$ is 
\[ \vdim (\ls) = \prod_{i=1}^n (d_i+1) - \sum_{i=1}^r \binom{ n+m_i-1}{n}-1\]
the {\em expected dimension} of $\ls$ is $\edim(\ls)=\max (\vdim(\ls),-1)$.
The dimension of $\ls$ is minimum when the points are in
very general position according to Remark~\ref{pivrp}.
The inequality $\dim (\ls)\geq \edim(\ls)$ always holds.
The conditions imposed by the points are 
linearly dependent if and only if  $\dim (\ls)> \edim(\ls)$,
in this case we say that $\ls$ is {\em special}.
Otherwise we say that $\ls$ is {\em non-special}.
Special linear systems have been 
classified when all the multiplicities are
$\leq 2$ in~\cite{CGG2,LP,La,VT}.

We prove that a fiber of a projection
map $(\pp^1)^n \rightarrow (\pp^1)^s$, where $1\leq s \leq n-1$,
through a multiple point can contribute to the speciality of the linear system
$\ls$. We introduce in Definition~\ref{def:fdim} the 
{\em fiber-expected dimension}
$\efdim(\ls)$ which satisfies the inequalities
(see Theorem~\ref{teo:fdim})
\[
 \dim(\ls)\geq\efdim(\ls)\geq\edim(\ls)
\]
and takes into account the speciality of $\ls$
coming from such fibers.
We say that $\ls$ is 
{\em fiber special}
if $\dim(\ls)>\fdim(\ls)$ and that it is 
{\em fiber non-special} otherwise.
In Theorem~\ref{fdim} we show that a 
linear system through two multiple points 
is fiber non-special.
If there are more than two multiple points
then there are examples of fiber special
systems (see Example~\ref{sp3}).
Linear systems through multiple points
of $(\pp^1)^n$ correspond to complete
linear systems on the blow-up $Y$ of
$(\pp^1)^n$ at those points.
We recall a quadratic form on $\Pic(Y)$
introduced in~\cite{Muk} and we study
the action of its Weyl group,
providing an algorithm for 
determining the element of minimal
total degree in any orbit of an effective
class. We say that such a class and the 
corresponding linear system are in 
{\em standard form}. For $n=3$
we state a conjecture about
special linear systems in standard
form.
Finally we recall a degeneration of $(\pp^1)^n$
into two copies of $(\pp^1)^n$ introduced in~\cite{LP}.
We relate the speciality of a linear system
with that of the two linear systems arising
from the degeneration.

The paper is organized as follows:
In Section~\ref{bs} we recall some definitions,
notations and introduce a small
modification between the blow-up of $\pp^n$
at $n+r-1$ points and the blow-up of $(\pp^1)^n$
at $r$ points. 
In Section~\ref{sbm} we
study the Weyl group of the Picard group 
of the blow-up of $(\pp^1)^n$ at points 
in very general position.
In Section~\ref{fss} we prove that
the fibers through the points can
contribute to the speciality of linear systems
and introduce the concept of fiber dimension. 
Section~\ref{dg} deals with a degeneration
of $(\pp^1)^n$ and the related algorithm.
Finally, in Section~\ref{ec}
we provide some examples
and state a conjecture for $(\pp^1)^3$.

\subsection*{Acknowledgements}
It is a pleasure to thank Elisa Postinghel
and Luca Ugaglia for several useful discussions.

\section{ Basic setup }\label{bs} 

In what follows we will denote by $\kk$
an algebraically closed field.
Given an algebraic variety $X$ we denote 
by $h^i(X,D)$ the dimension of the 
$i$-th cohomology group of any line bundle whose class is $D\in \Pic (X)$.

In this section we recall some definitions, 
notations and results about linear systems 
on $(\pp^1)^n$ and $\pp^n$.
First of all we denote by $\kk[x_1,y_1,\dots,x_n,y_n]$
the Cox ring of $(\pp^1)^n$
and by $\kk[x_0,\dots,x_n]$
the Cox ring of $\pp^n$.
Let $\pi_Y\colon Y \to (\pp^1)^n$
( resp. $\pi_X\colon X \to \pp^n$ )  
be the blowing-up at $r$ (resp. $r+n-1$) points in very general position.
The Picard Group of $Y$ is generated by the 
$r+n$ classes of $H_1, \dots , H_n, E_1, \dots E_r$ where 
$E_i$ is the exceptional divisor over the $i$-th point 
and $H_i$ is the pull-back of the prime divisor
of equation $x_i=0$. 
The Picard Group of $X$ is generated by the 
$r+n$ classes of $H, E_1, \dots, E_{r+n-1}$ where 
$E_i$ is the exceptional divisor over the $i$-th point 
and $H$ is the pull-back of a hyperplane.
We will call these bases {\em tautological}.

\begin{remark}(Points in very general position).\label{pivrp}
Let $q_1, \dots , q_r$ be distinct points of $(\pp^1)^n$
and let $m\in \nn^r$. Consider the scheme $(\pp^1)^n_{[r]}$
parametrizing $r$-tuples of points in $(\pp^1)^n$
and let $\mathcal{Q} \in (\pp^1)^n_{[r]}$ be the point
$q_1+\dots+q_r$.
For $d\in \mathbb{N}^n$
denote by $\mathcal{H}(d,m,\mathcal{Q})$ the vector space
of degree $d$ homogeneous polynomials
of $\kk[x_1, y_1, \dots , x_n, y_n]$ with multiplicity at least $m_i$ 
at each $q_i$. Observe that $\mathcal{H}(d,m, \mathcal{Q})$
depends on $\mathcal{Q}$ and that there is a Zariski open subset
$\mathcal{U}(d,m)\subset (\pp^1)^n_{[r]}$ where this dimension
attains its minimal value. Let us denote by 
\[ \mathcal{U} := \bigcap_{(d,m)\in \nn^{n+r}} \mathcal{U}(d,m).\]
We say that the points $q_1, \dots , q_r$ are in very general position
if the corresponding $\mathcal{Q}$ is in $\mathcal{U}$.
\end{remark}

\begin{definition}
Given a birational map $\phi : X \dashrightarrow Y$
of algebraic varieties we say that $\phi$ is a 
\textit{small modification} if there exist open subsets 
$U\subseteq X$ and $V\subseteq Y$ such that 
$\varphi(U)\subseteq V$, the restriction
$\phi|_U$ is an isomorphism and both
$X-U$ and $Y-V$ have codimension at lest two.
Note that any small modification induces 
mutually inverse isomorphisms of push-forward
and pull-back
\[
 \phi_* : \Pic (X) \rightarrow \Pic(Y)
 \qquad
 \phi^* : \Pic (Y) \rightarrow \Pic(X).
\]
Moreover $h^0 (Y,\phi_*(D))=h^0(X,D)$ for any 
$D\in \Pic (X)$ and $h^0 (X,\phi^*(D))=h^0(Y,D)$
for any $D\in\Pic(Y)$.
\end{definition}

\begin{definition}
Let $X$ be an algebraic variety, 
$D \in \Pic(X)$ a divisor class
and $V\subset X$ a subvariety.
We say that $V$ is contained with
multiplicity $m$ in the base locus of $D$
if the exceptional divisor $E$ of the blow-up
$\pi \colon \tilde{X}\rightarrow X$ 
of $X$ at $V$ is contained with multiplicity
$m$ in the base locus of $\pi^* (D)$.
\end{definition}

\begin{remark}\label{toric}
Let $\phi : \pp^n \rightarrow (\pp^1)^n$ be the
birational map defined by 
$[x_0 : \dots : x_n ]\mapsto ([x_{n-1}:x_n],\dots , [x_0:x_n])$.
Let $p_1,\dots,p_{r+n-1}$ be points of
$\pp^n$ in very general position such that
the first $n+1$ are the fundamental ones
and let $q_1,\dots,q_r$ be points of $(\pp^1)^n$
such that $q_1=([0:1], \dots , [0:1])$,  
$q_2=([1:0], \dots , [1:0])$ and $q_{i+2}
= \phi(p_{i+n+1})$ for $i\in\{1,\dots,r-2\}$.
This gives the following commutative diagram
\[
 \xymatrix{
 X \ar@{-->}[r]^-{\phi}\ar[d]^-{\pi_X}
  & Y \ar[d]^-{\pi_Y}\\
  \pp^n\ar@{-->}[r]^-{ \phi } & (\pp^1)^n,
 }
\]
where with abuse of notation we are denoting
by the same symbol $\phi$ and its lift.
To show that the above lift is a small modification
it is enough to consider the case $r=2$. In this
case we have commutative diagrams
\[
 \xymatrix@C=5pt{
  \Sigma_{n+1}^n\ar[d] & \supseteq & \Sigma 
  & \subseteq & \Sigma_2^{1,n}\ar[d]\\
  \Sigma^n & & & & \Sigma^{1,n}
 }
 \qquad
 \qquad
 \xymatrix@C=5pt{
   X_{n+1}^n\ar[d] & \supseteq & X(\Sigma)
  & \subseteq & Y_2^n\ar[d]\\
  \pp^n & & & & 
  (\pp^1)^n
 } 
\]

where the first diagram is obtained by completing
in two different ways the fan $\Sigma$ whose cones
are exactly the one-dimensional cones of $\mathbb{Z}^n$
generated by the vectors $\{\pm e_1,\dots,\pm e_n,
\pm(e_1+\dots+e_n)\}$, while the second diagram
is obtained applying the toric functor to the first one.
Since the complement of $X(\Sigma)$ in 
both $X(\Sigma_{n+1}^n)$ and $X(\Sigma^n)$
is of codimension at least two, then the corresponding
toric birational map $\phi\colon X\dashrightarrow Y$ 
is small.
We recall that the map $\phi$ and its action on
fat points has already been considered in~\cite{CGG}.
\end{remark}

With the above notation the induced isomorphism 
${\phi}_* : \Pic (X) \rightarrow \Pic (Y)$ 
is given by 
\begin{eqnarray}
\label{isometry}
\left\{
\begin{array}{lll}
 H & \mapsto \sum_{i=1}^n H_i -(n-1)E_1 \\[3pt]
 E_{n+1} & \mapsto E_2 \\[3pt]
 E_i & \mapsto H_{n+1-i}-E_1 & \text{ for } 1\leq i \leq n \\[3pt]
 E_i & \mapsto E_{i-n+1} & \text{ for } i> n+1.
\end{array}
\right.
\end{eqnarray}

\section{Standard form}\label{sbm}

Let us recall that we  
denote by $Y$ the blow-up
of $(\pp^1)^n$ at $r$ points 
$ q_1, \dots , q_r $
in very general position.
Without loss of generality we can assume the 
first two points to be $q_1=([0:1], \dots , [0:1])$,  
$q_2=([1:0], \dots , [1:0])$.
In this section we recall the definition
of a non-degenerate quadratic form 
$\Pic(Y)\to\zz$ already introduced
in~\cite{TC}.
Define the bilinear form
$\Pic(Y)\times\Pic(Y)\to\zz$
by $(D_1,D_2)\mapsto D_1\cdot D_2$
whose values on pairs of elements of
the basis $(H_1, \dots , H_n, E_1, \dots , E_r)$
are the following:
\begin{eqnarray}
\label{intersections}
 H_i\cdot H_j=1-\delta_{ij}
 \qquad
 E_k\cdot E_s=-\delta_{ks}
 \qquad
 H_i\cdot E_k = 0,
\end{eqnarray}
where $i,j\in\{1,\dots,n\}$ and
$k,s\in\{1,\dots,r\}$.
Observe that the lattice $\Pic(Y)$
equipped with the integer quadratic form induced
by the above bilinear form has
discriminant group isomorphic to
$\zz/(n-1)\zz$ and generated by
the class $\frac{1}{n-1}K_Y$.
Recall that given a non-degenerate lattice 
$\Lambda$ and an element $R\in\Lambda$
with $R^2=-2$ one can define the 
{\em Picard-Lefschetz reflection} defined
by $R$ as:
\[
\sigma_R \colon \Lambda\rightarrow\Lambda
\qquad
D\mapsto D+(D\cdot R)R.
\]
Observe that $\sigma_R$ is the reflection
in $\Lambda$ with respect to the hyperplane
orthogonal to $R$.
The {\em Weyl group} of $\Lambda$, denoted by
$W(\Lambda)$ is the subgroup of isometries 
of $\Lambda$ generated by the Picard-Lefschetz
reflections.
For simplicity, given an algebraic variety $X$
and a bilinear form on $\Pic (X)$
we denote the Weyl group of its Picard group
by $W(X)$.\\

The following is a particular case 
of ~\cite[Theorem 1]{Muk}:

\begin{proposition}\label{muk}
For each transformation 
$w\colon \Pic (Y)\rightarrow\Pic(Y)$
of $W(Y)$, there is a small modification 
$w \colon Y \dashrightarrow Y_w$
with the following property: 
$Y_w$ is also a blow-up of $(\pp^1)^n$
in $r$ points $q_1, \dots , q_r$ 
in general position and the pull-back 
of the tautological basis of $\Pic(Y_w)$ 
coincides with the transformation 
of the tautological basis of $Y$ by $w$.
\end{proposition}

In ~\cite[Lemma 2.1]{TC} 
a set of generators for $W(Y)$ consists 
of $n+r-1$ reflections with respect to the following roots:
\[
 H_1-E_1-E_2, \quad
 H_1-H_2,  \dots, H_{n-1}-H_n,\quad
 E_1-E_2, \dots , E_{r-1}-E_r.
\]
Let $\omega$ be any element of the Weyl 
group $W(Y)$ and let $\varphi_\omega\colon
Y\dashrightarrow Y_\omega$ be the corresponding
small modification. We have that $\varphi_\omega$
is the lift of a birational map $\phi_\omega\colon (\pp^1)^n
\to (\pp^1)^n$, moreover $Y_\omega$ is the 
blow-up of $(\pp^1)^n$ at points $q'_1, \dots , q'_r$ 
where $q'_1 =q_1$,$q'_2 =q_2$ and 
$q'_i=\phi_\omega(q_i)$ for $i\geq3$.
The birational involution of $(\pp^1)^n$
associated to the root $H_1-E_1-E_2$
is the following~\cite[pag. 128]{Muk}:
\[
([x_1:y_1],\dots,[x_n:y_n])
\mapsto
\left(\left[\frac{1}{x_1}:\frac{1}{y_1}\right], 
\left[ \frac{x_2}{x_1}:\frac{y_2}{y_1} \right],
  \dots, \left[\frac{x_n}{x_1}:\frac{y_n}{y_1}\right]\right).
\]
The birational involution of $(\pp^1)^n$
associated to the root $H_i-H_{i+1}$
is the transformations of $(\pp^1)^n$ 
which exchanges the $i$-coordinate 
with the $i+1$-coordinate, for $i\in\{1,\dots,n-1\}$.
Finally the birational map of $(\pp^1)^n$
associated to the root $E_i-E_{i+1}$
is the identity map as we are just
relabeling two points between the
$q_i$'s.

\begin{remark}\label{iii}
Observe that the map 
$\phi_* \colon \Pic (X) \rightarrow \Pic (Y)$
is an isometry of lattices.

To see this it is enough to check that $\phi_*$
preserves the intersection matrix of
the basis $(H,E_1,\dots,E_r)$ of $\Pic(X)$.
This holds by~\eqref{isometry}
and the definition of the bilinear forms on the two
lattices~\cite[2.1]{TC}.
We recall also that for any $\omega\in W(Y)$
and any $D,D'\in \Pic(Y)$ we have
$h^0(Y,D)=h^0(Y_\omega,\omega(D))$ 
by~\cite[Lemma 2.3]{TC} and
$D$ is integral if and only if $\omega(D)$ is.
\end{remark}

\begin{definition}
A class 
$D= \sum_{i=1}^n d_iH_i - \sum_{i}^r m_i E_i$
of $\Pic (Y)$
is in {\em pre-standard form} if the following
inequalities hold:
\[
 d_1 \geq d_2 \geq \dots \geq d_n \geq 0
 \qquad
 m_1 \geq m_2 \geq \dots \geq m_r
 \qquad
 \sum_{i=2}^n d_i \geq m_1 + m_2.
\]
If in addition $m_r\geq 0$,
then $D$ is in {\em standard form}.
\label{def4.1}
\end{definition}

\begin{remark}
By Definition~\ref{def4.1}
~\cite[Definition 3.1 ]{LU1} and the
action of $\phi$ given above
we have that a class $D$ in the
Picard group of $X$
is in pre-standard form (resp. in
standard form) if and only if 
${\phi}_*(D)$ is in pre-standard form
(resp. in standard form).
In particular by~\cite[Proposition 3.2]{LU1}
we deduce that for any effective class
$D\in \Pic(Y)$ there exists a $w\in W(Y)$ 
such that $w(D)$ is in pre-standard form.
\end{remark}

\begin{remark}
A $(-1)$-{\em class} of $\Pic(Y)$ is 
the class of an irreducible and reduced 
divisor $E$ such that $E^2= E\cdot K =-1$ 
where $K:=\frac{1}{n-1} K_{Y}$.
Observe that this definition coincides with
the classical concept of $(-1)$-class
when $n=2$.

By Remark~\ref{iii} and ~\cite[Section 4]{LU1} 
we conclude the following:
The $(-1)$-classes form an orbit with
respect to the action of the Weil group.
Moreover if $D$ is a class in standard form, 
then $w(D)\cdot E\geq 0$ for any $(-1)$-class 
$E$ and any $w\in W(Y)$.
Finally, some geometric properties of $(-1)$-curves on surfaces 
generalize to $(-1)$-classes:
if $D$ is effective and $D\cdot E <0$ for some $(-1)$-class $E$
then $E\subset Bs |D|$ 
and if $E, E'$ are two distinct $(-1)$-classes
having negative product with $D$ then $E\cdot E'=0$.\\
\end{remark}

The following program
given a class $D\in \Pic(Y)$
returns its standard form $D'\in \Pic(Y)$
or returns $0\in\Pic(Y)$ if the
linear system induced by $D$ is empty.\\

\begin{algorithm}[H]\label{alg1}
\caption{Standard form.}
 \KwIn{$(d,m)\in\nn^n\times\nn^r$, with $r\geq 2$.}
 \KwOut{$(d,m)\in\nn^{n}\times\nn^{r}$ or $\emptyset$.}
 Sort both $d=(d_1,\dots,d_n)$ and $m=(m_1,\dots,m_r)$ in decreasing order\;
\While{ $k:=\sum_{i=2}^{n} d_i - m_1-m_2<0$ and $\min(d_1, \dots , d_n)\geq 0$}
{                           
$(d_1,m_1,m_2):=(d_1,m_1,m_2)+(k,k,k)$\;
Sort both $d$ and $m$ in decreasing order\;
}
\eIf{ $\min(d_1, \dots , d_n)<0$ }{ \KwRet{ $\emptyset$}}{ \KwRet{ $(d,m)$ }\; }
\end{algorithm}

\section{Fiber special systems}\label{fss}

Recall that we denote by 
$\pi\colon Y\to (\pp^1)^n$ 
the blow-up of $(\pp^1)^n$ at $r$ points
$q_1, \dots , q_r $ in very general position.
Given a subset $I\subseteq \{1,\dots,n\}$
we denote by $P_I\colon (\pp^1)^n\to
(\pp^1)^{|I|}$ the morphism defined by 
(if $I$ is empty $P_I$ is the constant morphism
to a point)
\begin{eqnarray*}
 ([x_1 : y_1], \dots , [x_n : y_n])
 &\mapsto &
 ([x_i:y_i] : i\in I).
\end{eqnarray*}
We denote by $F_{j,I}$ the fiber
of $P_I$ through the point $q_j$ 
for any $j$.
Given a vector $(d_1, \dots , d_n)\in
\mathbb{N}^n$
we will denote by
\begin{eqnarray}
\label{SI}
 s_I 
 :=
 \sum_{i \in I} d_i
 \quad
 \text{ and }
 \quad
S_I := 1 + |I| +s_I
\quad
\text{ for any } I\subseteq \{1, \dots, n\},
\end{eqnarray}
where $s_\emptyset = 0$ and $S_\emptyset = 1$.
Observe that by the assumption made on the points
$F_{i,I}\cap F_{j,I}=\emptyset$ 
for any $i\neq j$.
In Section~\ref{fss} and Section~\ref{dg} we use the
notation $\ls:=\ls_{(d_1, \dots , d_n)}(m_1, \dots , m_r)$
to denote a general linear system when no confusion arises.
We denote by $V(\ls)$ the subvector space
of homogeneous polynomials
of $\kk[x_1,y_1,\dots,x_n,y_n]$
of degree $(d_1, \dots , d_{n})$
and multiplicity at least $m_{1}, \dots, m_r$ 
at $q_1, \dots, q_r$ respectively.

\begin{definition}
\label{def:fdim}
The \textit{fiber dimension} of the linear 
system $\ls$ is
\begin{eqnarray*}
 \fdim (\ls) 
 &:=&
 \prod_{i=1}^n (d_i+1) - 
 \sum_{\substack{1\leq j\leq r\\ 
 I\subseteq\{1,\dots,n\}\\ S_I \leq m_j} } (-1)^{|I|} \binom{  m_j -S_I+ n}{n}-1
\end{eqnarray*}
and the the {\em fiber-expected dimension}
is $\efdim(\ls):=\max(-1,\fdim(\ls))$.
We say that $\ls$ is {\em fiber special}
if $\dim(\ls)>\efdim(\ls)$ and it is 
{\em fiber non-special} otherwise.
\end{definition}

\begin{theorem}
\label{teo:fdim}
For any linear system $\ls$ the following inequalities
hold $\dim(\ls)\geq \efdim(\ls)\geq \edim(\ls)$.
\end{theorem}
\begin{proof}
Denote by $\Delta(m)\subseteq\zz_{\geq 0}^n$,
for $m\geq 1$, the set of integer points of the $n$-dimensional
simplex which is the convex hull of
the points: $0,(m-1)e_1,\dots,(m-1)e_n$.
Let $V\subseteq \kk[x_1,\dots,x_n]$ 
be the subvector space of polynomials
of degree at most $(d_1, \dots , d_n)$.
Given $w\in\zz_{z\geq 0}^n$ we define
the partial derivative $\partial/\partial x^w$,
where $x^w = x_1^{w_1}\cdots x_n^{w_n}$.
Let 
\[
 \Phi \colon V \rightarrow \kk^N
 \qquad
 f\mapsto \left(\dfrac{\partial f}{\partial x^w}(p_j)\, :\, 1\leq j\leq r
 \text{ and }w\in \Delta(m_j)\cap\zz_{\geq 0}^n\right)
\]
be the function which maps $f$ to the collection of 
all partial derivatives of $f$, corresponding 
to the integer points of the polytope 
$\Delta(m_j)$, 
evaluated at $p_j$ for each $j$.
Observe that $\dim(\ls)$ equals $\dim(\ker(\Phi))$.
Moreover if $w$ is an integer vector
outside the polytope $\Delta(m_j)\cap\prod_{i=1}^n[0,d_i]$,
then $\partial f/\partial x^w$ is the zero polynomial.
Thus any such $w$ does not impose
conditions on the kernel of $\Phi$.
Using the inclusion-exclusion principle
we see that the number of integer vectors
of the polytope $\Delta(m_j)\cap\prod_{i=1}^n[0,d_i]$
equals
\[
 \mu_j = \sum_{\substack{I\subseteq\{1,\dots,n\} \\ m_j \geq S_I} } 
 (-1)^{|I|} \binom{  m_j -S_I+ n}{n}.
\]
Thus the point $p_j$ of multiplicity $m_j$ can impose
at most $\mu_j$ conditions and the first inequality 
$\dim(\ls)\geq\efdim(\ls)$ follows.
The second inequality follows by observing
that the number of integer vectors of $\Delta(m_j)$
is greater than or equal to the number of integer
vectors of $\Delta(m_j)\cap\prod_{i=1}^n[0,d_i]$.
\end{proof}

\begin{remark}
For a linear system $\ls$ of $\pp^n$ through multiple base 
points in very general position, in~\cite{BDP} the authors 
introduce the liner expected dimension ${\rm eldim}(\ls)$,
which takes into account the speciality coming from
linear subspaces through some of the points. In that
case the authors asset that the inequality $\dim(\ls)
\geq {\rm eldim}(\ls)$ is equivalent to the weak Fr\"oberg-Iarrobino
conjecture~\cite[Remark 3.4]{BDP}.
The reason why in $(\pp^1)^n$ one can easily prove
the inequality $\dim(\ls)\geq\efdim(\ls)$ is that the 
each subvariety taken into account in the $\fdim$ formula
passes exactly through one point. 
\end{remark}

\begin{theorem}
\label{fdim}
A linear system $\ls$ through two points is fiber non-special.
\end{theorem}

The proof of the following lemma
is a direct consequence of the
identity $\sum_{i=n}^k \binom{i}{n} 
= \binom{k+1}{n+1}$ which holds
for any $k\geq n$.

\begin{lemma}\label{eqq}
Let $I$ be an ordered subset
of $\{1,\dots,n-1\}$, let 
$J:=I\cup\{n\}$ and let
$m$ be a non-negative integer.
Given a vector $(d_1, \dots , d_n)\in \nn^n$
let $S_I$ be defined as in~\eqref{SI}.
Then the following holds
\[
 \sum_{j=0}^{d_n}\sum_{m-j\geq S_I} 
 \binom{m-j-S_I+n-1}{n-1}
 =
 \sum_{m\geq S_I} \binom{m-S_I+n}{n}
 - \sum_{m\geq S_J} \binom{m- S_J+n}{n}.
\]
\end{lemma}

\begin{proof}[Proof of Theorem~\ref{fdim}]
Without loss of generality we can assume that 
$q_1:=([0:1],\dots, [0:1])$, $q_2:=([1:0],\dots , [1:0])$.
Hence a basis $\mathfrak{B}(\ls)$
for $V(\ls)$ consists of
the monomials of the form
$\prod_{i=1}^n x_i^{a_i}y_i^{b_i}$
where $\sum_{i=1}^n a_i \geq m_1$,
$\sum_{i=1}^n b_i \geq m_2$
and $a_i+b_i = d_i$
for any $i$.
The statement follows by induction on $n$
using Lemma~\ref{eqq} and the equality
\[
 |\mathfrak{B}(\ls)|
 =
 \sum_{j=0}^{d_n}
 |\mathfrak{B}(\ls_{(d_1,\dots,d_{n-1})}(m_1-j, m_2-d_n+j))|.
\]
\end{proof}

\begin{corollary}
A linear system $\ls:=\ls_{(d_1,\dots,d_n)}(m_1,m_2)$
is effective if and only if $\sum_{i=1}^nd_i\geq m_1+m_2$.
\end{corollary}
\begin{proof}
If $\sum_{i=1}^nd_i< m_1+m_2$ then, 
with the same notation of the proof of
Theorem~\ref{fdim} either $\sum_{i=1}^na_i< m_1$
or $\sum_{i=1}^nb_i< m_2$ so that there are
no monomials in $V(\ls)$ and thus $\ls$
is empty. On the other hand if 
$\sum_{i=1}^nd_i\geq m_1+m_2$
then there are $a_i,b_i$ such that 
$\sum_{i=1}^na_i\geq m_1$
and $\sum_{i=1}^nb_i\geq m_2$
and $a_i+b_i=d_i$ for any $i$.
Thus $V(\ls)$ contains a monomial
and hence $\ls$ is not empty.
\end{proof}

\begin{proposition}
Let $\ls$
be a non-empty linear system.
Then the fiber $F_{I,j}$
is contained in the base locus of $\ls$
with multiplicity
\[     \mu\geq \max \{ m_j-s_{I^c}, 0 \}       \]
and the equality holds when $r\leq 2$.
\end{proposition}

\begin{proof}
Without loss of generality we can assume that $j=1$
and $I=\{1, \dots i \}$.
Let $\mathcal{M}:=\ls_{(d_1,\dots,d_{n})}(m_1,m_2)$.
The vector space $V(\ls)$ is a subspace
of $V(\mathcal{M})$ which admits
a monomial basis $\mathcal{B}(\mathcal{M})$
given in the proof of Theorem~\ref{fdim}.
The Cox ring of the blow-up 
$\pi_{1,I}\colon X_{1,I}\rightarrow (\pp^1)^n$
of $(\pp^1)^n$ at $F_{1,I}$ is 
isomorphic to $\kk[zx_1, y_1, \dots,
zx_i, y_i , x_{i+1}, y_{i+1}, \dots, x_n, y_n]$,
where $z$ corresponds to 
the exceptional divisor.
Let $\mathcal{B}'(\mathcal{M})$ be the pull-back 
of the basis $\mathcal{B}(\mathcal{M})$ 
via $\pi_{1,I}$.
Then the basis $\mathcal{B}'(\mathcal{M})$ consists 
of the following monomials
\[
 \prod_{j=1}^i (zx_j)^{a_j}(y_j)^{b_j}
 \prod_{j=i+1}^n x_j^{a_j}y_j^{b_j},
\]
where $\sum_{i=1}^n a_i \geq m_1, \sum_{i=1}^n b_i\geq m_2$
 and $a_i+b_i=d_i$ for each $i$.
Observe that $\sum_{j=1}^i a_j
\geq m_1-\sum_{j=i+1}^n a_j\geq m_1-s_{I^c}$,
with equalities when $b_j=0$ for any $j\in \{ i+1, \dots , n\}$
and $\sum_{i=1}^n a_i =m_1$.
Thus $z^{m_1-s_{I^c}}$ divides any monomial in 
$\mathcal{B}'(\mathcal{M})$ and this is the maximal
power with this property when $\mathcal{M}=\ls$, 
i.e. when $r\leq 2$.
\end{proof}

\subsection{Base Locus of the Linear System}

In this subsection we describe a class of subvarieties
contained in the base locus $\Bs(\ls)$ of a linear system 
$\ls$ of the form $\mathcal{L}_{(d_1, \dots, d_n)}(m_1, \dots , m_r)$ 
of $(\pp^1)^n$ and compute their multiplicity in $\Bs(\ls)$.
By the generality assumption on the
position of the points we can assume all but the first two
of them to be contained in the $n$-dimensional
torus $\mathbb T^n$ of $(\pp^1)^n$. Denote by $p_{n+i-1} = \phi^{-1}(q_i)$,
for $3\leq i\leq r$, where $\phi\colon\pp^n\to (\pp^1)^n$
is the birational map defined in Remark~\ref{toric}.
Given two subsets $I\subseteq\{1,\dots,n\}$ and
$J\subseteq\{2,\dots,r\}$ we denote by $L_{IJ}$
the following linear subspace of $\pp^n$:
\[
 L_{IJ} = \langle\{e_i : i\in I\}\cup\{p_j : j\in J\}\rangle.
\]
We denote by $V_{IJ}$  the closure in $(\pp^1)^n$
of $\phi(L_{IJ} \cap\mathbb T^n)$.
Observe that if $L_{IJ}$ is defined by a matrix 
$A\in M_{k\times n}(\cc)$, with $k = |I| + |J| - 1$,
then $V_{IJ}\cap\phi(\mathbb T^n)$ is defined by the following
equations
\[
A \begin{bmatrix} 
y_1y_2\dots y_{n-1}x_n\\
y_1y_2\dots x_{n-1}y_n\\
\vdots \\
x_1y_2\dots y_{n-1}y_n\\
y_1y_2\dots y_{n-1}y_n\\
\end{bmatrix} =
\begin{bmatrix} 
0\\
0\\
\vdots \\
0\\
0\\
\end{bmatrix}.
\]

\begin{proposition}
\label{VIJ}
Let $\ls=\mathcal{L}_{(d_1, \dots ,d_n)}(m_1, \dots ,m_r)$
be a non-empty linear system and let $V_{IJ}$ be as above.
Then $V_{IJ}$ is contained in $\Bs(\ls)$ with multiplicity
\[
\mu_{IJ} := \max \left\{ 0 , (|J|-1)(m_1-\delta) 
+ \sum_{i\in J}m_i 
- \sum_{i\in I} d_i \right\},
\]
where we denote by $\delta=\sum_{i=1}^n d_i$.
\end{proposition}

\begin{proof}
The multiplicity of $V_{IJ}$ in the base locus of 
$\ls$ equals the multiplicity of
$L_{IJ}$ in the base locus of $\phi^*(\ls)$.
By~\eqref{isometry} the class
of an element of $\phi^*(\ls)$ is
\[
\left( \delta  -m_1 \right)H - \sum_{i=1}^n \left( \delta -m_1-d_{n-i} \right)E_i
-\sum_{i=2}^r m_{n+i-1} E_{n+i-1}.
\]
Thus we conclude by~\cite{BDP}*{Proposition 2.5}.
\end{proof}

As a  consequence of the fact that the base locus
of a linear system through $n+2$ points in $\pp^n$ is
a union of linear subspaces
~\cite{BDP}*{Corollary 4.8} we immediately get
the following.

\begin{corollary}\label{3pts}
If $\mathcal{L}_{(d_1, \dots , d_n)}(m_1, m_2, m_3)$
is a linear system through three points
in general position, then its base locus only 
contains varieties of the form $F_{I,j}$ and $V_{IJ}$. 
\end{corollary}

\section{ Degeneration of $(\pp^1)^n$}\label{dg}

In this section we use the degeneration in $(\pp^1)^n$
introduced in ~\cite[Section 3]{LP} and using a method 
introduced in ~\cite{Dum} we prove a Theorem
that allows to check non-speciality of a linear system in $(\pp^1)^n$
via this degeneration.

Recall that $\Delta(m+1)\subseteq\zz_{\geq 0}^n$,
for $m\geq 0$, the set of integer points of the $n$-dimensional
simplex which is the convex hull of
the points: $0,me_1,\dots,me_n$.

$\ls_{(d_1, \dots , d_{n-1},k\rightarrow d_n)}(m_1, \dots , m_r)$
will denote the sublinear system of  $\ls$
defined by all the polynomials divisible by $x_n^k$.
We denote by $V_A(\ls)$ the subvector space
of $\kk[x_1,\dots,x_n]$ obtained by evaluating
the polynomials of $V(\ls)$ at $y_1=\dots=y_n=1$.
Observe that $V_A(\ls)$ is the subvector space
of polynomials $f\in \kk[x_1, \dots , x_n]$ 
of degree at most $(d_1, \dots , d_n)$
such that 
$f$ has multiplicity at least $m_j$ at $p_j$
for any $j$.
Let $V:=V_A(\ls_{(d_1, \dots , d_n)})$
and let
\[
 \Phi \colon V \rightarrow \kk^N
\]
be the function which maps $f$ into the collection of 
all partial derivatives of $f$, which correspond 
to the integer points of the polytope 
$\Delta(m_i)\cap\prod_{i=1}^n[0,d_i]$ 
evaluated at $p_i$ for each $i$
(see also the proof
of Theorem~\ref{teo:fdim}).
Let $M(\ls)$ be the matrix of $\Phi$ with
respect to the monomial basis of $V$
and the standard basis of $\kk^N$.
The columns $M(\ls)$ are indexed by monomials of 
degree at most $(d_1, \dots , d_n)$, while
rows are indexed by conditions imposed by the points.
Let $P=\kk[p^1_1, \dots , p^n_1, \dots , p_r^1, \dots , p_r^n]$,
where $p^i_k$ is the $i$-coordinate of the $k$-th point.
Then the entries of $M(\ls)$ can be considered as polynomials in $P$.
Let $s$ be a positive integer $\leq r$, let
$\deg$ be a grading on $P$ defined by
$\deg(p_j^k)=1$ if $k=n$ and $j\geq s+1$
and $\deg(p_j^k)=0$ otherwise.
In what follows we will adopt the following
notation:

\begin{equation}
\small
\label{l12}
 \ls_1:=\ls_{(d_1, \dots , d_{n-1}, k-1)}  (m_1, \dots , m_s)
 \quad
 \ls_2:=\ls_{(d_1, \dots , d_{n-1}, d_n-k)} ( m_{s+1}, \dots, m_r).
\end{equation}

\begin{theorem}\label{dgt}
Let $\ls_1$ and $\ls_2$ be defined as in~\eqref{l12}.
Assume that the following conditions hold:
\begin{enumerate}
\item
$\ls_1$, $\ls_2$ are fiber non-special
with $(\fdim (\ls_1) +1)(\fdim (\ls_2)+1)\geq 0$,
\item
$m_i\leq k$, for any $i\in\{1,\dots,s\}$,
\item $m_j\leq d_n-k+1$ for any $j\in\{s+1,\dots,n\}$.
\end{enumerate}
Then the system $\ls:=\ls_{(d_1, \dots , d_n)}(m_1, \dots , m_r)$
is fiber non-special.
\end{theorem}

\begin{proof}
Observe that we have an isomorphism of
vector spaces
\[
 \Psi\colon
 V_A( \ls_2)
 \to
 V_A(\ls_{(d_1, \dots , d_{n-1},k\rightarrow d_n)}(m_{s+1}, \dots , m_r))
\]
where the multiplicities are imposed at the points
$p_{s+1}, \dots , p_r$ respectively.
After reordering the rows and the columns
of the matrix $M( \ls )$ we can assume that
its first %$\gamma:=k\prod_{i=1}^{n-1}(d_i+1)$ 
$\gamma$ columns 
are indexed by monomials 
of degree at most $(d_1, \dots , d_{n-1}, k-1)$
and that its first $\rho$
rows are indexed by conditions imposed 
at the points $p_1, \dots , p_s$.
We write
\[ 
M(\ls) = \begin{bmatrix} 
M_1 & K_1 \\
K_2 & M_2 \\
\end{bmatrix},
\]
where $M_1$ is a $\rho\times\gamma$ matrix.
Observe that $M_1 = M(\ls_1 )$ 
and $M_2 \cong M( \ls_2 )$ via the
isomorphism $\Psi$. Moreover by conditions
$(2)$, $(3)$ and the fact that $\ls_1$,
$\ls_2$ are fiber non-special, we deduce
that both matrices have maximal rank.
Assume now that
$\fdim (\ls_1) \geq -1$ and $\fdim (\ls_2) \geq -1$
(the other case being analysed in
a similar way).
Choose two submatrices $M_i'$ of
$M_i$ of maximal rank, for $i\in\{1,2\}$,
and form the square submatrix of $M(\ls)$
\[ M' = \begin{bmatrix} 
M'_1 & K'_1 \\
K'_2 & M'_2 \\
\end{bmatrix}
\]
where $K_1'$ is obtained from 
$K_1$ by deleting columns of $M_2$
and similarly for $K_2'$.
By ~\cite[Lemma 2]{Dum}
we have that $\deg(\det(M'_2))>\deg(\det(B))$
for any square submatrix $B$
of $[K'_2\ \ M'_2]$.
Thus, by the Laplace expansion
with respect to the first $\rho$ rows 
we conclude that $\deg (\det ( M' )) =
\deg(\det(M'_1)\cdot\det(M'_2))>0$ and
the result follows.
\end{proof}

The following algorithm 
is a recursive program that uses
Theorem~\ref{dgt} and Theorem~\ref{fdim}
in order to conclude if the given
linear system is non-special. \\

\begin{algorithm}[H]\label{alg2}
 \KwIn{$(d,m)\in\nn^n\times\nn^r$, with $r\geq 2$.}
 \KwOut{$x\in\{\text{non-special, undecided, special}\}$.}

 \uIf{$\std(d,m)=\emptyset$}{\KwRet{{\rm non-special}.}}
 \uElseIf{$\fdim(\std(d,m)) > \edim(d,m)$}{\KwRet{{\rm special}}\;}
 \Else{$(d,m):=\std(d,m)$\;

 \eIf{ $r=2$}
 { \eIf{ $\fdim(d,m) \geq \edim (d,m) $ }{\KwRet{{\rm special}}\;}{\KwRet{{\rm non-special}}\; } }{
\For{ $k\in \{ 1, \dots , d_1-1\}$, $s\in \{ 1, \dots , r-1\}$ }{
$d' := (k-1, d_2, \dots ,d_n ) $, $m':=(m_1,\dots, m_s)$\;
$d'':=(d_1-k , d_2, \dots, d_n) $, $m'':=(m_{s+1}, \dots , m_r)$\;
\If{ 
$\seep(d',m')=${\rm non-special} and $\seep(d'',m'')=${\rm non-special}\\
and $(\fdim(d',m')+1)(\fdim(d'',m'')+1)\geq 0$\\
and $m_i\leq k$ for any $i\in\{1,\dots,s\}$\\
and $m_j\leq d_n-k$ for any $j\in\{s+1,\dots,r\}$\\
}{\KwRet{{\rm non-special}}\;}
}
\KwRet{{\rm undecided}}\;
}
}
\caption{Speciality by degeneration.}
\end{algorithm}

\section{ Examples and conclusions}\label{ec}

We have studied linear systems of $(\pp^1)^n$
passing through points in 
very general position
and concluded that the
fibers of the projections $(\pp^1)^n\to (\pp^1)^k$,
for $1\leq k < n$, can contribute to the speciality.
The following is an example
of a fiber special linear system $\ls$
whose standard form $\ls'$ is 
fiber non-special.

\begin{example}\label{ex1}
The linear system $\ls:=\ls_{(13,9,5)}(11^2,7^2,3^2)$
of $(\pp^1)^3$ is not in standard form with
\[ \vdim(\ls)= 12^2\cdot 8^2\cdot 4^2 - 2\left( 
\binom{13}{3} + \binom{9}{3} + \binom{5}{3} \right)=80 \qquad \fdim(\ls)=154. \]
Using Algorithm~\ref{alg1}
we obtain the following linear systems
\[
 \ls_{(13,9,5)}(11^2,7^2,3^2)
 \leadsto
 \ls_{(5,9,5)}(7^2,3^4)
 \leadsto 
 \ls_{(5,5,5)}(3^6) =:\ls'
\]
where $\ls'$ is in standard form.
Algorithm~\ref{alg2} degenerates $\ls'$
according to the following scheme:
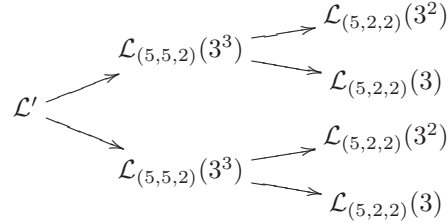
\[
 \xymatrix@R=-5pt{
  & & \ls_{(5,2,2)}(3^2)\\
  & \ls_{(5,5,2)}(3^3)\ar[ru]\ar[rd]\\
  & & \ls_{(5,2,2)}(3)\\
  \ls'\ar[ruu]\ar[rdd] & & \\
  & & \ls_{(5,2,2)}(3^2)\\
  & \ls_{(5,5,2)}(3^3)\ar[ru]\ar[rd]\\
  & & \ls_{(5,2,2)}(3)
 }
\]
By Theorem~\ref{fdim} the last four 
linear systems are non-special,
thus by repeated applications
of Theorem~\ref{dgt} we conclude
that $\ls'$ is non-special as well.
In particular $\dim(\ls)=\dim(\ls')
=\vdim(\ls')=156$.
\end{example}

The following example shows that
there are other varieties giving contribution
to the speciality of the linear system
already when we blow-up three
points in very general position.

\begin{example}\label{sp3}
The linear system 
$\ls := \ls_{(1,1,1,1,1,1,1)}(3^3)$
of $(\pp^1)^7$ is in standard form
with
 \[ \vdim(\ls)=2^7-3\binom{9}{7} = 20 \qquad \fdim(\ls)= \vdim(\ls)+21 =41,\]
where the contribution 
on the right is given by the $21$
one-dimensional fibers on the base locus,
but we have that $\dim(\ls)=42$,
then $\ls$ is fiber-special. 
Observe that Algorithm~\ref{alg2}
returns {\em undecided} in this case
since every degeneration gives a
special linear system.
The dimension of $\ls$ 
can be calculated by evaluating directly
the rank of the matrix $M(\ls)$, appearing
in the proof of Theorem~\ref{dgt}. 
By Corollary~\ref{3pts}
the base locus of $\ls$ is the union 
of all the fibers through 
each of the three points plus 
the irreducible surfaces $V_{IJ}$ 
for $J=\{2,3\}$ and $I=\{i\}\subseteq \{1,\dots,7\}$,
plus the curve $C = V_{\emptyset,\{2,3\}}$.
By Proposition~\ref{VIJ} each $V_{IJ}$ is contained
in the base locus of $\ls$ with multiplicity $1$ and
$C$ is contained with multiplicity $2$.
Moreover the equality
\[
 \dim(\ls) = \fdim(\ls) + 1
\]
suggests that $C$ is contributing to the speciality
of $\ls$.
\end{example} 

\begin{remark}
Observe that the strict inequality $\dim(\ls) > \efdim(\ls)$
can hold also in the simple case when all the
multiplicities equal $2$. For instance the linear
system $\ls = \ls_{(2,2,2)}(2^7)$ is special of dimension
$0$ and $\efdim(\ls) = -1$. The subvariety of
$(\pp^1)^3$ which produces the speciality
is the unique surface of the linear system
$\ls_{(1,1,1)}(1^7)$. For a complete classification
of the base loci of special linear systems through
double points of $(\pp^1)^n$ see~\cite[Section 7]{LP}.
\end{remark}

Denote, as before, by $Y$ the blow-up of 
$(\pp^1)^3$ at $r$ points in very general position
and by $\phi\colon\pp^3\to (\pp^1)^3$ the birational
map defined in Remark~\ref{toric}.
Let $Q$ be a divisor in the strict
transform of the linear system
$\ls_{(1,1,1)}(1^7)$ which 
is the image via $\phi^*$ of the class of the
strict transform of the quadric through $9$ points
of $\pp^3$.
For any divisor $D$ in the strict transform of
$\ls_{(d_1,d_2,d_3)}(m_1,\dots,m_7)$
let 
\[ q(D):=\chi(D|_Q)=(d_1+1)(d_2+1)(d_3+1) -d_1d_2d_3 - \sum_{i=1}^7 \dfrac{ m_i(m_i+1)}{2}.\]

The following conjecture
is equivalent to ~\cite[Conjecture 6.3]{LU1}
via the small modification $\phi$.

\begin{conjecture}
Let $\ls:=\ls_{(d_1,d_2,d_3)}(m_1,\dots,m_r)$
be a linear system in standard form and
let $D$ be a divisor in its strict transform.
\begin{itemize}
\item If $q(D)\leq 0$, then $h^0(D)=h^0(D-Q)$.
\item If $q(D)>0$, then $D$ is special if and only if $m_1>d_n+1$ 
and $D$ is fiber non-special.
\end{itemize}
\end{conjecture}

\begin{example}
Let $\ls_n = \ls_{(n,n,n)}(n^7)$
be the linear system corresponding to the divisor
class $nQ\in \Pic(Y)$, where $n>0$.
This system has dimension $1$ for
any $n$ and it is non-special
for $n=1$. Its fiber dimension is 
\[ \fdim( \ls_n ) = \vdim (\ls_n) = n^3 - 7 \binom{ 2n-1}{ n } < 0,\]
so that $\ls_n$ is fiber-special for $n>1$.
It is easy to check that $q(nQ)= 0$ for any $n$
and that in this case the conjecture holds.
\end{example}

\end{document}